\documentclass{amsart}

\usepackage{amsmath}
\usepackage{amssymb}
\usepackage{latexsym}
\usepackage{amsbsy}
\usepackage[all]{xy}
\usepackage{amscd}
\usepackage{mathrsfs}
\usepackage[dvips]{graphicx}

\newtheorem{thm}{Theorem}[section]
\newtheorem{prop}[thm]{Proposition}
\newtheorem{lem}[thm]{Lemma}
\newtheorem{cor}[thm]{Corollary}

\theoremstyle{definition}

\theoremstyle{remark}
\newtheorem{rem}[thm]{Remark}

\numberwithin{equation}{section}

\DeclareMathOperator{\soc}{soc}
\DeclareMathOperator{\tp}{top}
\DeclareMathOperator{\dimv}{\underline{dim}}

\DeclareMathOperator{\Irr}{Irr}

\DeclareMathOperator{\nil}{nil}
\DeclareMathOperator{\GL}{GL}

\DeclareMathOperator{\id}{id}
\DeclareMathOperator{\op}{op}

\newcommand\bfi{\mathbf{i}}

\newcommand\bfa{\mathbf{a}}

\begin{document}

\title[Parametrizations of canonical bases and nilpotent varieties]{Parametrizations of canonical bases and irreducible components of nilpotent varieties}

\author{Yong Jiang}

\address{Fakult\"{a}t f\"{u}r Mathematik, Universit\"{a}t Bielefeld, Postfach 10 01 31, D-33501 Bielefeld, Germany}

\email{yjiang@math.uni-bielefeld.de}

\thanks{Supported by the Sonderforschungsbereich 701 "Spectral Structures and Topological Methods in Mathematics" in Bielefeld University.}





\begin{abstract}
It is known that the set of irreducible components of nilpotent varieties provides a geometric realization of the crystal basis for quantum groups. For each reduced expression of a Weyl group element, Gei{\ss}, Leclerc and Schr\"{o}er has recently given a parametrization of irreducible components of nilpotent varieties in studying cluster algebras. In this paper we show that their parametrization coincides with Lusztig's parametrization of the canonical basis.
\end{abstract}

\maketitle

\section{Introduction}\label{sec 1}
Let $\mathfrak{g}$ be a Kac-Moody algebra associated with a symmetric Cartan matrix and $U_{q}(\mathfrak{g})$ be its quantized enveloping algebra. Denote by $\mathscr{B}(\infty)$ the crystal basis \cite{Ka1} of the negative part $U_{q}^{-}(\mathfrak{g})$. The globalization of the crystal basis coincides with Lusztig's canonical basis \cite{Lu1}\cite{Lu2}.

If $\mathfrak{g}$ is finite-dimensional, the canonical basis of $U_{q}^{-}(\mathfrak{g})$ has a labeling of $r$-tuple of non-negative integers through the PBW-basis \cite{Lu1}, where $r$ is the length of the longest element $w_{0}$ in the Weyl group. More precisely, for each reduced expression $\bfi$ of $w_{0}$, we have a bijection $\psi_{\bfi}:\mathbb{N}^{r}\simeq\mathscr{B}(\infty)$, known as Lusztig's parametrization. It was later generalized to the case of Kac-Moody algebras in the following sense \cite{Lu4}: For each $w$ of length $r$ and a reduced expression $\bfi$, we have an injective map $\psi_{\bfi}:\mathbb{N}^{r}\hookrightarrow\mathscr{B}(\infty)$. The image does not depend on the choice of $\bfi$ and thus can be denoted by $\mathscr{B}(w)$ (see \cite{Ki}). There are many interesting applications of Lusztig's parametrizations of canonical bases, e.g. criteria of total positivity \cite{BFZ} and combinatorial expressions of tensor product multiplicities \cite{BZ}.

On the other hand, Kashiwara and Saito \cite{KS} gave a geometric construction of the crystal $\mathscr{B}(\infty)$ using Lusztig's nilpotent varieties, which are varieties of certain modules over preprojective algebras. The set $\mathcal{B}$ of irreducible components of nilpotent varieties has a crystal structure isomorphic to $\mathscr{B}(\infty)$. Recently Gei{\ss}, Leclerc and Schr\"{o}er has shown in \cite{GLS2} that for any $w$ in the Weyl group, a subset (will be denoted by $\mathcal{B}(w)$ in \ref{subsec 3 GLS parametrization}) of $\mathcal{B}$ provides the dual semicanonical basis \cite{Lu5} of the coordinate ring $\mathbb{C}[N(w)]$ of the corresponding unipotent subgroup. For each reduced expression $\bfi$ of $w$, elements in $\mathcal{B}(w)$ are Zariski closures of certain irreducible constructible subsets $\Lambda_{\bfi}^{\bfa}$ ($\bfa\in\mathbb{N}^{r}$). And $\Lambda_{\bfi}^{\bfa}$ is the set of $\Lambda$-modules filtered by certain modules $M_{\bfi,k}$ ($1\leq k\leq r$) with multiplicities $\bfa$ (see \ref{subsec 3 GLS parametrization}). In this way we have an $\mathbb{N}^{r}$-parametrization of $\mathscr{B}(\infty)$. In other words, we have an injective map $\phi_{\bfi}:\mathbb{N}^{r}\hookrightarrow\mathcal{B}$ and the image $\mathcal{B}(w)$ only depends on $w$.

Thus it is quite natural to study the relationship between Lusztig's and GLS's parametrizations. Moreover, we may ask if the subset $\mathcal{B}(w)$ is the geometric counterpart of $\mathscr{B}(w)$, i. e. the image of $\mathscr{B}(w)$ under the crystal isomorphism $\mathscr{B}(\infty)\simeq\mathcal{B}$ is precisely $\mathcal{B}(w)$.

In this paper we give an affirmative answer to the above question by showing that the two parametrizations actually coincide, if we identify $\mathscr{B}(\infty)$ with its geometric counterpart $\mathcal{B}$ (Theorem \ref{thm main}). This result also gives an explicit description of the isomorphism $\mathscr{B}(\infty)\simeq\mathcal{B}$ restricting to the subset $\mathscr{B}(w)$ (see Remark \ref{rem importance of thm 2}).

If $\mathfrak{g}$ is finite-dimensional and $\bfi$ is a $Q$-adapted reduced expression of the longest element $w_{0}$, our result has been proved by Baumann and Kamnitzer \cite{BK}. However, our theorem works for any reduced expression $\bfi$ of $w\in W$ and is new even in the finite type case (see Remark \ref{rem importance of thm 1}).

In the proof of our main result we use reflection functors for preprojective algebras introduced in \cite{BK}. We prove that all the modules $M_{\bfi,k}$, whose multiplicities yield GLS's parametrization, can be obtained from simple $\Lambda$-modules by applying reflection functors (Proposition \ref{prop M_k and replection funtctor}). This result was essentially proved in \cite{AIRT}, where the modules $M_{\bfi,k}$ were introduced differently in a dual form and some derived reflection functors were used. The relationship between the two classes of functors has been discussed in \cite{BKT}, which appeared in arXiv shortly after our paper. However, we present a new and direct proof. And our approach shows that another family of modules $V_{\bfi,k}$ appearing in \cite{GLS2} is also closely related to reflection functors.

After a preliminary version of this paper had been written, we were informed by Jan Schr\"{o}er that Bolten \cite{Bo} also introduced reflection functors for preprojective algebras, very similar to Baumann and Kamnitzer's, and proved the fact that the modules $M_{\bfi,k}$ can be obtained via reflection functors.

The paper is organized as follows: In Section \ref{sec 2} and Section \ref{sec 3} we provide preliminaries on quantum groups, nilpotent varieties as well as the two parametrizations of crystals. We then prove the fact that the modules $M_{\bfi,k}$ can be obtained from simple modules by applying reflection functors in Section \ref{sec 4}. Finally, in Section \ref{sec 5} we prove the main result of this paper, namely the coincidence of Lusztig's parametrization of the canonical basis and GLS's parametrization of irreducible components of nilpotent varieties.

\section{Quantum groups and parametrizations of canonical bases}\label{sec 2}

\subsection{Basic notions}\label{subsec 2 basic notions}
Let $\Gamma$ be a finite graph without loops and $I=\{1,2,\ldots,n\}$ be the set of vertices. Let $\mathfrak{g}$ be the (symmetric) Kac-Moody algebra associated with $\Gamma$. Denote by $\alpha_{i}$ (resp. $\varpi_{i}$) the simple roots (resp. fundamental weights). Let $P$ be the weight lattice. Let $W$ be the corresponding Weyl group generated by simple reflections $s_{i}$ ($i\in I$). The length of an element $w\in W$ is denoted by $\ell(w)$. If $w=s_{i_{r}}s_{i_{r-1}}\cdots s_{i_{1}}$ is a reduced expression we say $\bfi=(i_{r},i_{r-1}\ldots,i_{1})$ is a reduced expression of $w$.

Let $U_{q}(\mathfrak{g})$ be the quantized enveloping algebra of $\mathfrak{g}$, which is a $\mathbb{Q}(q)$-algebra with generators $e_{i}$, $f_{i}$ ($i\in I$) and $q^{h}$ ($h\in P^{\ast}$). Let $U_{q}^{-}(\mathfrak{g})$ be the subalgebra of $U_{q}(\mathfrak{g})$ generated by $f_{i}$ ($i\in I$). Note that the defining relations of $U_{q}^{-}(\mathfrak{g})$ are the quantum Serre relations.

\subsection{Quantum unipotent subgroups}\label{subsec 2 quantum unipotent subgroups}
For each $i\in I$, Lusztig defined $\mathbb{Q}(q)$-algebra automorphisms $T_{i,e}',T_{i,e}''$ ($e\in\{\pm 1\}$) of $U_{q}(\mathfrak{g})$ (see \cite{Lu3} for details). We will write $T_{i}$ (resp. $T_{i}^{-1}$) for $T_{i,1}''$ (resp. $T_{i,-1}'$).

For any $w\in W$ and a reduced expression $\bfi=(i_{r},i_{r-1},\ldots,i_{1})$, set $\beta_{\bfi,k}=s_{i_{1}}\cdots s_{i_{k-1}}(\alpha_{i_{k}})$, for any $1\leq k\leq r$.

We then define the following quantum root vectors:
$$F(\beta_{\bfi,k})=T_{i_{1}}^{-1}\cdots T_{i_{k-1}}^{-1}(f_{i_{k}}), \quad \text{for all} \ 1\leq k\leq r.$$

Note that these elements are in $U_{q}^{-}(\mathfrak{g})$. For $\bfa=(a_{1},\ldots,a_{r})\in\mathbb{N}^{r}$ set
$$F_{\bfi}(\bfa)=F(\beta_{\bfi,r})^{(a_{r})}\cdots F(\beta_{\bfi,1})^{(a_{1})}.$$

The subspace of $U_{q}^{-}(\mathfrak{g})$ spanned by $\{F_{\bfi}(\bfa)|\bfa\in\mathbb{N}^{r}\}$ is a subalgebra of $U_{q}^{-}(\mathfrak{g})$ and independent of the choice of $\bfi$ (see \cite{Lu3}). Thus we can denote it by $U_{q}^{-}(w)$, called the \textit{quantum unipotent subgroup} associated with $w$ (see \cite{Ki}).

Moreover, the set $\mathcal{P}_{\bfi}=\{F_{\bfi}(\bfa)|\bfa\in\mathbb{N}^{r}\}$ is a basis of $U_{q}^{-}(w)$, called the PBW-basis attached to $\bfi$. In particular, if $\mathfrak{g}$ is finite dimensional and $w=w_{0}$ the longest element in the Weyl group, we have $U_{q}^{-}(w_{0})=U_{q}^{-}(\mathfrak{g})$ and $\mathcal{P}_{\bfi}$ is a basis of $U_{q}^{-}(\mathfrak{g})$.

\subsection{Lusztig's parametrization of the canonical basis}\label{subsec 2 parametrization of crystal}
Let $\mathscr{L}(\infty)$ (resp. $\mathscr{B}(\infty)$ be the crystal lattice (resp. crystal basis) of $U_{q}^{-}(\mathfrak{g})$. We refer to \cite{Ka1} for all missing definitions.

\begin{thm}
Let $w\in W$ and $\bfi$ be a reduced expression of $w$.

(i). For any $\bfa=(a_{1},\ldots,a_{r})\in\mathbb{N}^{r}$, there exists a unique element $b_{\bfi,\bfa}\in\mathscr{B}(\infty)$ such that
$$F_{\bfi}(\bfa)\equiv b_{\bfi,\bfa} \mod\ q\mathscr{L}(\infty). $$

(ii). The map $\psi_{\bfi}:\mathbb{N}^{r}\rightarrow\mathscr{B}(\infty)$ defined by $\bfa\mapsto b_{\bfi,\bfa}$ is injective.

(iii). The image of $\psi_{\bfi}$ does not depend on the choice of $\bfi$.
\end{thm}

In the above theorem, (i) and (ii) were proved in \cite{Lu4} (We use the crystal basis instead of the canonical basis just for convenience). (iii) can be seen from a recent result in \cite{Ki} asserting that the quantum unipotent subgroup $U_{q}^{-}(w)$ is compatible with the dual canonical basis.

Thus we can denote by $\mathscr{B}(w)$ the image of $\psi_{\bfi}$. The theorem gives an $\mathbb{N}^{r}$-parametrization of elements in $\mathscr{B}(w)$. For each $b\in\mathscr{B}(w)$, $\psi_{\bfi}^{-1}(b)$ is called the \textit{$\bfi$-Lusztig data} of $b$.

Note that when $\mathfrak{g}$ is finite-dimensional and $w=w_{0}$, the map $\bfa\mapsto\ b_{\bfi,\bfa}$ is a bijection and $\mathscr{B}(w)=\mathscr{B}(\infty)$.

\section{Preprojective algebras, nilpotent varieties and reflection functors}\label{sec 3}

\subsection{Preprojective algebras and nilpotent varieties}\label{subsec 3 preprojective algebras}
Let $\Lambda$ be the preprojective algebra associated to the graph $\Gamma$ (see for example \cite{R}). Denote by $\bmod\Lambda$ the category of finite dimensional $\Lambda$-modules. For each $i\in I$ we have a one-dimensional simple $\Lambda$-module $S_{i}$ concentrated on the vertex $i$. In general there exist other finite-dimensional simple modules. A $\Lambda$-module $M$ is called \textit{nilpotent} if there is a composition series of $M$ with all factors of the form $S_{i}$ ($i\in I$). Let $\nil(\Lambda)$ be the full subcategory of $\bmod(\Lambda)$ consisting of finite-dimensional nilpotent $\Lambda$-modules.

We will identify the dimension vector of a $\Lambda$-module $M$ as an element in the root lattice by setting $\dimv S_{i}=\alpha_{i}$. For a $\Lambda$-module $M$, the \textit{$i$-socle} (resp. \textit{$i$-top}) of $M$ is the $S_{i}$-isotypic component of the socle (resp. top) of $M$, denoted by $\soc_{i}M$ (resp. $\tp_{i}M$).

For a dimension vector $\nu=\sum_{i\in I}\nu_{i}\alpha_{i}\in Q_{+}$, let $\Lambda(\nu)$ be the variety of finite-dimensional nilpotent $\Lambda$-modules with dimension vector $\nu$. It is an affine algebraic variety. Recall that the algebraic group $\GL(\nu)=\prod_{i\in I}\GL_{\nu_{i}}(\mathbb{C})$ acts on $\Lambda(\nu)$ such that two points in $\Lambda(\nu)$ are in the same orbit if and only if they are isomorphic as $\Lambda$-modules. These varieties were first studied by Lusztig (see \cite{Lu1}) and usually called \textit{nilpotent varieties}.

Denote by $\Irr\Lambda(\nu)$ the set of irreducible components of the variety $\Lambda(\nu)$. Let $\mathcal{B}=\bigsqcup_{\nu\in Q_{+}}\Irr\Lambda(\nu)$. Kashiwara and Saito associated a crystal structure on $\mathcal{B}$ and proved that there is a (unique) crystal isomorphism $\Psi:\mathscr{B}(\infty)\rightarrow\mathcal{B}$ (see \cite{KS} for details).

\subsection{GLS's parametrization of irreducible components}\label{subsec 3 GLS parametrization}
For each $i\in I$, let $\widehat{I}_{i}$ be the indecomposable injective $\Lambda$-module with socle $S_{i}$. Note that these modules are infinite-dimensional if $\Gamma$ is not of type $ADE$.

For a sequence $(j_{1},\ldots,j_{t})$ of indices with $1\leq j_{p}\leq n$ for all $p$, there is a unique chain
$$0=X_{0}\subseteq X_{1}\subseteq \cdots \subseteq X_{t}\subseteq X$$
of submodules of $X$ such that $X_{p}/X_{p-1}=\soc_{j_{p}}(X/X_{p-1})$. Define $\soc_{(j_{1},\ldots,j_{t})}(X):=X_{t}$.

Let $\bfi=(i_{r},\ldots,i_{1})$ be a reduced expression of $w\in W$. For $1\leq k\leq r$, set
$$V_{\bfi,k}=\soc_{(i_{k},\ldots,i_{1})}(\widehat{I}_{i_{k}}).$$

For $1\leq k\leq r$, let $k^{-}=\max\{0,1\leq s\leq k-1|i_{s}=i_{k}\}$. Then for each $1\leq k\leq r$ there is a canonical embedding $\iota_{k}:V_{\bfi,k^{-}}\rightarrow V_{\bfi,k}$, where for $k^{-}=0$ we set $V_{\bfi,0}=0$.

Let $M_{\bfi,k}$ be the cokernel of $\iota_{k}$. For $\bfa=(a_{r},\ldots,a_{1})\in\mathbb{N}^{r}$, let $\Lambda_{\bfi}^{\bfa}$ be the set of all $\Lambda$-modules $X$ such that there exists a chain
$$0=X_{0}\subseteq X_{1}\subseteq \cdots\subseteq X_{r}=X$$
such that $X_{k}/X_{k-1}\cong M_{\bfi,k}^{a_{k}}$ for all $1\leq k\leq r$.

It is clear that all the modules in $\Lambda_{\bfi}^{\bfa}$ have the same dimension vector $\mu(\bfa)=\sum_{k=1}^{r}a_{k}\dimv M_{\bfi,k}$. Hence $\Lambda_{\bfi}^{\bfa}$ is a subset of the variety $\Lambda(\mu(\bfa))$.

\begin{prop}[\cite{GLS2}]\label{prop parametrization of irr cpt}
$\Lambda_{\bfi}^{\bfa}$ is an irreducible constructible subset of $\Lambda(\mu(\bfa))$ and the Zariski closure $Z_{\bfi}^{\bfa}$ of $\Lambda_{\bfi}^{\bfa}$ is an irreducible component. In particular, $Z_{\bfi}^{\bfa}$ is the unique irreducible component of $\Lambda(\mu(\bfa))$ which contains a dense open subset belonging to $\Lambda_{\bfi}^{\bfa}$.
\end{prop}

Thus we have an injective map $\phi_{\bfi}:\mathbb{N}^{r}\rightarrow\mathcal{B}$ given by $\bfa\mapsto Z_{\bfi}^{\bfa}$. And from \cite{GLS2} we can see that the image of the map $\phi_{\bfi}$ is independent of the choice of $\bfi$, since it provides the dual semicanonical basis of the coordinate ring of the unipotent subgroup $N(w)$. So we can denote the image of $\phi_{\bfi}$ by $\mathcal{B}(w)$.

\subsection{Reflection functors for preprojective algebras}\label{subsec 3 ref funct for preproj alg}
For $i\in I$, the reflection functor $\Sigma_{i}$ ($\Sigma^{\ast}_{i}$) for the preprojective algebra is a natural generalization of the BGP-reflection functor for a corresponding quiver with respect to a sink (resp. source) vertex. For precise definitions we refer to \cite{BK}.

We collect some basic properties of reflection functors in the following lemma.
\begin{lem}[\cite{BK}]\label{lem prop of ref funt}
(i). $\Sigma_{i}$ is left exact and $\Sigma_{i}^{\ast}$ is right exact.

(ii). We have the following functorial short exact sequences:
\begin{gather*}
0\rightarrow\soc_{i}\rightarrow \id\rightarrow\Sigma_{i}\Sigma_{i}^{\ast}\rightarrow 0,\\
0\rightarrow\Sigma_{i}^{\ast}\Sigma_{i}\rightarrow \id\rightarrow\tp_{i}\rightarrow 0.
\end{gather*}

(iii). Let $i$ and $j$ be two vertices of $Q$ such that they are linked by one single arrow. Then the functors $\Sigma_{i}\Sigma_{j}\Sigma_{i}$ and $\Sigma_{j}\Sigma_{i}\Sigma_{j}$ are isomorphic.

(iv). If $\tp_{i}M=0$ (resp. $\soc_{i}M=0$), then $\dimv \Sigma_{i}M=s_{i}(\dimv M)$ (resp. $\dimv \Sigma^{\ast}_{i}M=s_{i}(\dimv M)$).
\end{lem}

\section{The modules $M_{\bfi,k}$ via reflection functors}\label{sec 4}
In this section we give a direct proof of the fact that the modules $M_{\bfi,k}$ can be obtained from simple $\Lambda$-modules by applying reflection functors, which is crucial in the proof of our main results. We will use the modules $N(w\varpi_{i})$ defined in \cite{BK} and show that they give an alternative construction of the modules $V_{\bfi,k}$.

\subsection{The modules $N(w\lambda)$}\label{subsec 4 N(w_lambda)}
Let $\widehat{\Gamma}$ be the graph obtained from $\Gamma$ by adding a vertex $i'$ and an edge $d_{i}$ connecting $i$ and $i'$ for each vertex $i\in I$. Then we have the associated preprojective algebra $\widehat{\Lambda}$.

It is convenience to write the dimension vector of a $\widehat{\Lambda}$-module $M$ as a pair $(\mu,\lambda)\in Q_{+}\times P_{+}$. That is, $\mu=\sum_{i\in I}\mu_{i}\alpha_{i}$ and $\lambda=\sum_{i\in I}\lambda_{i}\varpi_{i}$, where $\mu_{i}=\dim M_{i}$ and $\lambda_{i}=\dim M_{i'}$.

Let $\lambda\in P_{+}$ be a dominant weight. Define $\widehat{N}(\lambda)$ to be the $\widehat{\Lambda}$-module with dimension vector $(0,\lambda)$, which is unique up to isomorphism.

For each $w\in W$ with $\ell(w)\geq 1$, we define
$$\widehat{N}(w\lambda)=\Sigma_{i_{r}}\Sigma_{i_{r-1}}\cdots\Sigma_{i_{1}}\widehat{N}(\lambda).$$
where $\bfi=(i_{r},\ldots,i_{1})$ is a reduced expression of $w$. Note that by Lemma \ref{lem prop of ref funt} (iii), $\widehat{N}(w\lambda)$ is independent of the choice of $\bfi$ and thus well-defined.

We have a canonical embedding $\widehat{N}(w\lambda)\hookrightarrow\widehat{N}(s_{i}w\lambda)$ if $\ell(s_{i}w)>\ell(w)$ (see \cite{BK} Section 3.4). In particular, $\widehat{N}(\lambda)$ is a submodule of $\widehat{N}(w\lambda)$. We define $N(w\lambda)$ to be the quotient $\widehat{N}(w\lambda)/\widehat{N}(\lambda)$.

By definition we know that applying $\Sigma_{i}$ for any $i\in I$ to a $\widehat{\Lambda}$-module $M$ does not change the vector space $M_{j'}$ for any $j$. Thus the underlying space of $\widehat{N}(\lambda)$ is exactly the sum of underlying spaces of $\widehat{N}(w\lambda)$ at vertices $\{j'|j\in I\}$. Therefore $N(w\lambda)$ is a $\Lambda$-module.

The following results will be used (see \cite{BK} Theorem 3.1 and Theorem 3.4):
\begin{lem}\label{lem prop of N hat}
(i). For $i\in I$ such that $\ell(s_{i}w)>\ell(w)$, $\widehat{N}(w\lambda)$ has trivial $i$-top.

(ii). For any $i\in I$ and $1\neq w\in W$, $\soc N(w\varpi_{i})=S_{i}$ and $\dimv N(w\varpi_{i})=\varpi_{i}-w\varpi_{i}$.
\end{lem}

\subsection{An alternative construction of $V_{\bfi,k}$}\label{subsec 4 V_k coincides}
From now on we fix $w\in W$ and a reduced expression $\bfi=(i_{r},\ldots,i_{1})$. Recall that $V_{\bfi,k}=\soc_{(i_{k},\ldots,i_{1})}\widehat{I}_{i_{k}}$ (see \ref{subsec 3 GLS parametrization}). It was stated in \cite{BK} without proof that the modules $N(w\lambda)$ had been studied in \cite{GLS2}. We give a precise statement and present a proof in this subsection.

For a dominant weight $\lambda=\sum_{i\in I}\lambda_{i}\varpi_{i}$, we define an injective $\Lambda$-module $\widehat{I}^{\lambda}:=\oplus_{i\in I}\widehat{I}_{i}^{\lambda_{i}}$.

\begin{lem}\label{lem savage lem}
For each $w\in W$, there is a unique (up to isomorphism) submodule of $\widehat{I}^{\lambda}$ with dimension vector $\lambda-w\lambda$.
\end{lem}

\begin{proof}
Let $Gr(\lambda-w\lambda,\widehat{I}^{\lambda})$ be the projective variety consisting of all $\Lambda$-submodules of $\widehat{I}^{\lambda}$ with dimension vector $\lambda-w\lambda$. In \cite{Sav2} it was proved that the variety $Gr(\lambda-w\lambda,\widehat{I}^{\lambda})$ is homeomorphic to the Lagrangian quiver variety $\mathcal{L}(\lambda-w\lambda,\lambda)$. Then the statement in the lemma is just a reformulation of Proposition 5.1 in \cite{Sav1}, which asserts that the variety $\mathcal{L}(\lambda-w\lambda,\lambda)$ is a point.
\end{proof}

Now it is easy to prove the following result:
\begin{prop}\label{prop V coincide with N}
For any $1\leq k\leq r$, $V_{\bfi,k}\cong N(s_{i_{1}}s_{i_{2}}\cdots s_{i_{k}}\varpi_{i_{k}})$.
\end{prop}
\begin{proof}
 By Lemma \ref{lem prop of N hat}, we know that $N(s_{i_{1}}s_{i_{2}}\cdots s_{i_{k}}\varpi_{i_{k}})$ is a submodule of $\widehat{I}_{i_{k}}$ with dimension vector $\varpi_{i_{k}}-s_{i_{1}}s_{i_{2}}\cdots s_{i_{k}}\varpi_{i_{k}}$. By definition and \cite{GLS2} Corollary 9.2 we know that $V_{\bfi,k}$ is also a submodule of $\widehat{I}_{i_{k}}$ with the same dimension vector. Thus they have to be isomorphic by the previous lemma.
\end{proof}

\subsection{Reflection functors and the modules $M_{\bfi,k}$}\label{subsec 4 M_k and refl funct}
We first prove an easy result on reflection functors. Let $\bmod(\Lambda)[i]$ (resp. $\bmod(\Lambda)[i]^{\ast}$) be the subcategory of $\bmod(\Lambda)$ consisting of modules with trivial $i$-top (resp. $i$-socle). Lemma \ref{lem prop of ref funt} (ii) implies that $\Sigma_{i}$ and $\Sigma_{i}^{\ast}$ give inverse equivalences of categories $\bmod(\Lambda)[i]\rightleftarrows\bmod(\Lambda)[i]^{\ast}$. In general the functor $\Sigma_{i}$ is not right exact and $\Sigma_{i}^{\ast}$ is not left exact. But we have the following result:

\begin{lem}\label{lem restrict exactness}
The restriction of $\Sigma_{i}$ (resp. $\Sigma_{i}^{\ast}$) on $\bmod(\Lambda)[i]$ (resp. $\bmod(\Lambda)[i]^{\ast}$) is exact.
\end{lem}
\begin{proof}
We only prove the statement for $\Sigma_{i}$. The one for $\Sigma_{i}^{\ast}$ can be proved similarly.

Suppose we have the following short exact sequence in $\bmod(\Lambda)[i]$:
$$0\rightarrow M_{1}\rightarrow M_{2}\rightarrow M_{3}\rightarrow 0.$$

Since $\Sigma_{i}$ is left exact, we have the exact sequence
\begin{equation}\label{equ 1}
\Sigma_{i}M_{1}\rightarrow \Sigma_{i}M_{2}\rightarrow \Sigma_{i}M_{3}\rightarrow 0.
\end{equation}

By Lemma \ref{lem prop of ref funt} (iv), we have
$$\dimv(\Sigma_{i}M_{j})=s_{i}(\dimv M_{j}),\quad \text{for } j=1,2,3.$$

So $\dimv(\Sigma_{i}M_{2})=\dimv(\Sigma_{i}M_{1})+\dimv(\Sigma_{i}M_{3})$, which forces (\ref{equ 1}) to be a short exact sequence.
\end{proof}

Next we show how to get the modules $M_{\bfi,k}$ via reflection functors.

\begin{prop}\label{prop M_k and replection funtctor}
For any $1\leq k\leq r$, $M_{\bfi,k}\cong\Sigma_{i_{1}}\Sigma_{i_{2}}\cdots\Sigma_{i_{k-1}}S_{i_{k}}$.
\end{prop}

\begin{proof}
First we have the following short exact sequence:
$$0\rightarrow \widehat{N}(\varpi_{i_{k}})\rightarrow \widehat{N}(s_{i_{k}}\varpi_{i_{k}})\rightarrow S_{i_{k}}\rightarrow 0.$$
since $\widehat{N}(s_{i_{k}}\varpi_{i_{k}})=\Sigma_{i_{k}}\widehat{N}(\varpi_{i_{k}})$ and $\widehat{N}(\varpi_{i_{k}})=S_{i_{k}'}$ .

It is clear that the three modules occurring in the above sequence have trivial $i_{k-1}$-top. So applying the functor $\Sigma_{i_{k-1}}$ and using the previous lemma, we have the short exact sequence
\begin{equation*}
0\rightarrow \widehat{N}(s_{i_{k-1}}\varpi_{i_{k}}) \rightarrow \widehat{N}(s_{i_{k-1}}s_{i_{k}}\varpi_{i_{k}})\rightarrow \Sigma_{i_{k-1}}S_{i_{k}}\rightarrow 0.
\end{equation*}

Now by Lemma \ref{lem prop of N hat} (i), $\widehat{N}(s_{i_{k-1}}s_{i_{k}}\varpi_{i_{k}})$ and $\widehat{N}(s_{i_{k-1}}\varpi_{i_{k}})$ both have trivial $i_{k-2}$-top. So $\Sigma_{i_{k-1}}S_{i_{k}}$ also has trivial $i_{k-2}$-top. Hence we can repeat the above procedure. Namely, applying $\Sigma_{i_{k-2}},\ldots,\Sigma_{i_{1}}$ successively, we have the following short exact sequence
\begin{equation}\label{equ 2}
0\rightarrow \widehat{N}(s_{i_{1}}\cdots s_{i_{k-1}}\varpi_{i_{k}})\rightarrow \widehat{N}(s_{i_{1}}\cdots s_{i_{k}}\varpi_{i_{k}})\rightarrow \Sigma_{i_{1}}\cdots\Sigma_{i_{k-1}}S_{i_{k}}\rightarrow 0.
\end{equation}

Note that for any $l$ such that $k^{-}<l<k$, we have $i_{l}\neq i_{k}$. This implies $\widehat{N}(s_{i_{l}}\varpi_{i_{k}})=\Sigma_{i_{l}}\widehat{N}(\varpi_{i_{k}})=\widehat{N}(\varpi_{i_{k}})$ because the module $\widehat{N}(\varpi_{i_{k}})$ is concentrated at the vertex $i_{k}'$, which is not connected with any other vertex except $i_{k}$. So we have
$$\Sigma_{i_{1}}\Sigma_{i_{2}}\cdots\Sigma_{i_{k-1}}\widehat{N}(\varpi_{i_{k}})=\Sigma_{i_{1}}\Sigma_{i_{2}}\cdots\Sigma_{i_{k^{-}}}\widehat{N}
(\varpi_{i_{k}}).$$

Thus the sequence (\ref{equ 2}) is the following
\begin{equation}\label{equ 3}
0\rightarrow \widehat{N}(s_{i_{1}}s_{i_{2}}\cdots s_{i_{k^{-}}}\varpi_{i_{k}})\rightarrow \widehat{N}(s_{i_{1}}s_{i_{2}}\cdots s_{i_{k}}\varpi_{i_{k}})\rightarrow \Sigma_{i_{1}}\Sigma_{i_{2}}\cdots\Sigma_{i_{k-1}}S_{i_{k}}\rightarrow 0.
\end{equation}

Now $\widehat{N}(s_{i_{1}}s_{i_{2}}\cdots s_{i_{k^{-}}}\varpi_{i_{k}})$ and $\widehat{N}(s_{i_{1}}s_{i_{2}}\cdots s_{i_{k}}\varpi_{i_{k}})$ both have the submodule $\widehat{N}(\varpi_{i_{k}})$. And the map $\widehat{N}(s_{i_{1}}s_{i_{2}}\cdots s_{i_{k}}\varpi_{i_{k}})\rightarrow \Sigma_{i_{1}}\Sigma_{i_{2}}\cdots\Sigma_{i_{k-1}}S_{i_{k}}$ clearly maps $\widehat{N}(\varpi_{i_{k}})$ to zero. So (\ref{equ 3}) yields
$$0\rightarrow N(s_{i_{1}}s_{i_{2}}\cdots s_{i_{k^{-}}}\varpi_{i_{k}})\rightarrow N(s_{i_{1}}s_{i_{2}}\cdots s_{i_{k}}\varpi_{i_{k}})\rightarrow \Sigma_{i_{1}}\Sigma_{i_{2}}\cdots\Sigma_{i_{k-1}}S_{i_{k}}\rightarrow 0.$$

Applying Proposition \ref{prop V coincide with N} (note that $i_{k}=i_{k^{-}}$) we have
$$0\rightarrow V_{\bfi,k^{-}}\rightarrow V_{\bfi,k}\rightarrow \Sigma_{i_{1}}\Sigma_{i_{2}}\cdots\Sigma_{i_{k-1}}S_{i_{k}}\rightarrow 0.$$

Hence $\Sigma_{i_{1}}\Sigma_{i_{2}}\cdots\Sigma_{i_{k-1}}S_{i_{k}}\cong V_{\bfi,k}/V_{\bfi,k^{-}}=M_{\bfi,k}$.
\end{proof}

The proof of the above theorem implies the following corollary, which will be used in the next section:
\begin{cor}\label{cor trivial top}
For any $l<k-1$, the module $\Sigma_{i_{l+1}}\cdots\Sigma_{i_{k-1}}S_{i_{k}}$ has trivial $i_{l}$-top.
\end{cor}

\subsection{Remarks on the adaptable case}\label{subsec 4 adaptable}
Let $Q$ be a quiver with underlying graph $\Gamma$. For $i\in I$, denote by $\sigma_{i}Q$ the quiver obtained from $Q$ by reversing all the arrows connected with $i$, if $i$ is a sink or a source.

A reduced expression $\bfi=(i_{r},\ldots,i_{1})$ of $w\in W$ is called \textit{$Q$-adapted} if $i_{1}$ is a sink of $Q$ and $i_{k}$ is a sink of $\sigma_{i_{k-1}}\cdots\sigma_{i_{1}}Q$ for all $2\leq k\leq r$. An element $w\in W$ is called \textit{adaptable} if there exists a quiver $Q$ and a reduced expression such that $\bfi$ is $Q$-adapted.

As pointed out in \cite{GLS2}, if $w$ is adaptable and $\bfi$ is $Q^{\op}$-adapted, the module $M_{\bfi}$ is a \textit{terminal $\mathbb{C}Q$-module} in the sense of \cite{GLS1}. This means that the modules $M_{\bfi,k}$ ($1\leq k\leq r$) are certain indecomposable preinjective $\mathbb{C}Q$-modules. In particular, if $\mathfrak{g}$ is finite-dimensional, the longest element $w_{0}\in W$ is always adaptable. If $\bfi$ is a $Q^{\op}$-adapted reduced expression of $w_{0}$, the set $\{M_{\bfi,k}|1\leq k\leq r\}$ forms a complete set of pairwise non-isomorphic indecomposable $\mathbb{C}Q$-modules. But in general there does not exist any quiver $Q$ such that all the modules $M_{\bfi,k}$ are $\mathbb{C}Q$-modules.

Note that for any quiver $Q$, the indecomposable preprojective and preinjective modules can be obtained from simple modules via BGP-reflection functors (see \cite{BGP} Theorem 1.3). Thus our Proposition \ref{prop M_k and replection funtctor} can be viewed as a generalization of this classical result to the case of any $w\in W$ and any reduced expression $\bfi$.

\section{Compatibility of two parametrizations}\label{sec 5}
Throughout this section, we fix $w\in W$ and a reduced expression $\bfi=(i_{r},\ldots,i_{1})$.

\subsection{The main result}\label{subsec 5 state of main thm}
Recall that we have injective maps $\psi_{\bfi}:\mathbb{N}^{r}\rightarrow\mathscr{B}(\infty)$ (see \ref{subsec 2 parametrization of crystal}), $\phi_{\bfi}:\mathbb{N}^{r}\rightarrow\mathcal{B}$ (see \ref{subsec 3 GLS parametrization}) and the Kashiwara-Saito crystal isomorphism $\Psi:\mathscr{B}(\infty)\simeq\mathcal{B}$.

\begin{thm}\label{thm main}
The following diagram is commutative:
\begin{equation*}
\xymatrix{
\mathbb{N}^{r} \ar[r]^-{\psi_{\bfi}} \ar[dr]_-{\phi_{\bfi}} & \mathscr{B}(\infty) \ar[d]^{\Psi}\\
                              & \mathcal{B}
}
\end{equation*}
i.e. with the notations in \ref{subsec 2 parametrization of crystal} and \ref{subsec 3 GLS parametrization}, we have $\Psi(b_{\bfi,\bfa})=Z_{\bfi}^{\bfa}$, for any $\bfa\in\mathbb{N}^{r}$.
\end{thm}

The theorem will be proved in the next subsection. From the theorem it follows immediately that the subset $\mathcal{B}(w)$ is the geometric counterpart of the unipotent crystal $\mathscr{B}(w)$:

\begin{cor}\label{cor geom descriotion of unipt crystal}
$\Psi(\mathscr{B}(w))=\mathcal{B}(w)$.
\end{cor}

\begin{rem}\label{rem importance of thm 1}
In the case that $\mathfrak{g}$ is finite-dimensional, $w=w_{0}$ the longest element in the Weyl group and $\bfi$ a $Q$-adapted reduced expression, the result in Theorem \ref{thm main} has been proved in \cite{BK} Proposition 7.8. In fact, in this case the set of modules $M_{\bfi,k}$ ($1\leq k\leq r$) is exactly a complete set of pairwise non-isomorphic indecomposable representations of $Q$. For any $\bfa\in\mathbb{N}^{r}$, set $M_{\bfi}^{\bfa}=\oplus_{k=1}^{r}M_{\bfi,k}^{a_{k}}\in\bmod(\mathbb{C}Q)$. It is not difficult to see that $\Lambda_{\bfi}^{\bfa}=T^{\ast}\mathcal{O}_{M_{\bfi}^{\bfa}}$, the conormal bundle of the orbit of $M_{\bfi}^{\bfa}$. So the irreducible component $Z_{\bfi}^{\bfa}$ is the same as the closure of $T^{\ast}\mathcal{O}_{M_{\bfi}^{\bfa}}$.

The arguments in \cite{BK} also work for any Kac-Moody algebra $\mathfrak{g}$ and $Q$-adapted reduced expression $\bfi$ (which means that $w$ is $Q$- adaptable). However, even in the case of finite type, there exists non-adaptable $w$ in the Weyl group (For example, consider type $D_{4}$ with $2$ being the central vertex and $w=s_{1}s_{2}s_{3}s_{2}$). Thus Theorem \ref{thm main} is new in call cases.
\end{rem}

\begin{rem}\label{rem importance of thm 2}
In general, for any $b\in\mathscr{B}(\infty)$, the image $\Psi(b)\in\mathcal{B}$ is not easy to describe. We can only use the fact that $\Psi$ commutes with $\widetilde{f}_{i}$ and keeps the unique highest weight element. More precisely, denote by $b_{0}$ (resp. $Z_{0}$) the unique highest weight element in $\mathscr{B}(\infty)$ (resp. $\mathcal{B}$). One need to find a path in the crystal graph from $b_{0}$ to $b$, say, $b=\widetilde{f}_{j_{1}}\widetilde{f}_{j_{2}}\cdots\widetilde{f}_{j_{s}}b_{0}$. Then we have $\Psi(b)=\widetilde{f}_{j_{1}}\widetilde{f}_{j_{2}}\cdots\widetilde{f}_{j_{s}}Z_{0}$.

We see that Theorem \ref{thm main} gives an explicit description of the image $\Psi(b)$ for any $b\in\mathscr{B}(w)$ once we know the $\bfi$-Lusztig data of $b$ for some $\bfi$.
\end{rem}

\subsection{Proof of theorem \ref{thm main}}\label{subsec 5 proof of thm}
First we recall some definitions. In \cite{Ka1, Ka2} it was proved that $\mathscr{B}(\infty)$ admits an involution $\ast$ induced by an algebra involution on $U_{q}(\mathfrak{g})$. And we have the $\ast$-Kashiwara operators $\widetilde{e}_{i}^{\ast}=\ast\circ\widetilde{e}_{i}\circ\ast$, $\widetilde{f}_{i}^{\ast}=\ast\circ\widetilde{f}_{i}\circ\ast$. For any $b\in\mathscr{B}(\infty)$, $\widetilde{e}_{i}^{\max}(b):=\widetilde{e}_{i}^{\varepsilon_{i}(b)}(b)$, $\widetilde{e}_{i}^{\ast\max}(b):=\widetilde{e}_{i}^{\ast\varepsilon_{i}^{\ast}(b)}(b)$. In \cite{Sai}, Saito introduced operators $\mathcal{T}_{i}$ and $\mathcal{T}_{i}^{-1}$ (originally denoted by $\Lambda_{i}$ and $\Lambda_{i}^{-1}$) on $\mathscr{B}(\infty)$ as follows
$$\mathcal{T}_{i}(b):=\widetilde{f}_{i}^{\ast\varphi_{i}(b)}\widetilde{e}_{i}^{\max}(b),\quad \mathcal{T}_{i}^{-1}(b):=\widetilde{f}_{i}^{\varphi_{i}^{\ast}(b)}\widetilde{e}_{i}^{\ast\max}(b).$$
They are analogues of Lusztig's automorphism $T_{i},T_{i}^{-1}$ at the crystal level.

Now we show how to deduce the $\bfi$-Lusztig data of any $b\in\mathscr{B}(w)$ by applying operators $\widetilde{e}_{i}^{\ast}$ and $\mathcal{T}_{i}$.

\begin{prop}\label{prop determine Lusztig data}
Let $b\in\mathscr{B}(w)$ and assume that $\psi_{\bfi}^{-1}(b)=\bfa=(a_{r},\ldots,a_{1})$. Then

(i). $a_{1}=\varepsilon_{i_{1}}^{\ast}(b)$. $\psi_{\bfi}^{-1}(\widetilde{e}_{i}^{\ast\max}b)=(a_{r},\ldots,a_{2},0)$.

(ii). Let $w'=s_{i_{r}}\cdots s_{i_{2}}$ and denote by $\bfi'=(i_{r},\ldots,i_{2})$, $\bfa'=(a_{r},\ldots,a_{2})$. Suppose that $a_{1}=0$, then $\mathcal{T}_{i_{1}}(b)\in\mathscr{B}(w')$ and we have $\psi_{\bfi'}^{-1}(\mathcal{T}_{i_{1}}(b))=\bfa'$.
\end{prop}

\begin{proof}
We know that
\begin{equation*}
b\equiv F_{\bfi}(\bfa)=T_{i_{1}}^{-1}\cdots T_{i_{r-1}}^{-1}(f_{i_{r}}^{(a_{r})})\cdots T_{i_{1}}^{-1}(f_{i_{2}}^{(a_{2})})f_{i_{1}}^{(a_{1})}, \mod q\mathscr{L}(\infty).
\end{equation*}

Write $P=T_{i_{1}}^{-1}\cdots T_{i_{r-1}}^{-1}(f_{i_{r}}^{(a_{r})})\cdots T_{i_{1}}^{-1}(f_{i_{2}}^{(a_{2})})$. So $F_{\bfi}(\bfa)=Pf_{i_{1}}^{(a_{1})}$.

Note that $P\in T_{i_{1}}^{-1}(U_{q}^{-}(\mathfrak{g}))\cap U_{q}^{-}(\mathfrak{g})$, by \cite{Sai} Proposition 2.1.2 and the definition of $\widetilde{e}_{i}$, we have $\widetilde{e}_{i}^{\ast}(P)=0$ and $F_{\bfi}(\bfa)=\widetilde{f}_{i}^{\ast a_{1}}(P)$. Hence $\widetilde{e}_{i}^{\ast\max}F_{\bfi}(\bfa)=P\equiv \widetilde{e}_{i}^{\ast\max}(b) \mod q\mathscr{L}(\infty)$. This proves (i).

Now assume $a_{1}=0$, we have $b\equiv F_{\bfi}(\bfa)=P\in T_{i_{1}}^{-1}(U_{q}^{-}(\mathfrak{g}))\cap U_{q}^{-}(\mathfrak{g})$. So
$$T_{i_{1}}(P)= T_{i_{2}}^{-1}\cdots T_{i_{r-1}}^{-1}(f_{i_{r}}^{(a_{r})})\cdots T_{i_{2}}^{-1}(f_{i_{3}}^{(a_{3})})f_{i_{2}}^{(a_{2})}=F_{\bfi'}(\bfa').$$

By \cite{Sai} Proposition 3.4.7, $\mathcal{T}_{i_{1}}(b)\equiv T_{i_{1}}(P)\mod q\mathscr{L}(\infty)$, which yields (ii).
\end{proof}

\begin{rem}
In \cite{Sai} Proposition 3.4.7 it was assumed that $P=G(b)$ the canonical basis element corresponding to $b$. However, for the proof there one only needs $P\equiv b\mod q\mathscr{L}(\infty)$ and $P\in\mathscr{L}(\infty)$.
\end{rem}

Using the above proposition we have $\mathcal{T}_{i_{1}}\widetilde{e}_{i_{1}}^{\ast\max}(b_{\bfi,\bfa})=b_{\bfi',\bfa'}$. We can repeat the procedure and finally we will reach the unique highest weight element $b_{0}$. This gives the following corollary.

\begin{cor}\label{cor reach b from the hwc}
$\widetilde{e}_{i_{r}}^{\ast\max}\mathcal{T}_{i_{r-1}}\widetilde{e}_{i_{r-1}}^{\ast\max}\cdots\mathcal{T}_{i_{1}}\widetilde{e}_{i_{1}}^{\ast\max}
(b_{\bfi,\bfa})=b_{0}$.
\end{cor}

Now we are ready to prove the main theorem. As in Proposition \ref{prop determine Lusztig data} we denote $\bfi'=(i_{r},\ldots,i_{2})$ and $\bfa'=(a_{r},\ldots,a_{2})$.

\begin{lem}\label{lem Sigma star X}
For any $X\in\Lambda_{\bfi}^{\bfa}$, $\Sigma_{i_{1}}^{\ast}X\in\Lambda_{\bfi'}^{\bfa'}$.
\end{lem}
\begin{proof}
By definition $X$ has a filtration
\begin{equation}\label{equ 4}
0=X_{0}\subseteq X_{1}\subseteq \cdots \subseteq X_{r}=X
\end{equation}
such that $X_{k}/X_{k-1}\cong M_{\bfi,k}^{a_{k}}$ for any $1\leq k\leq r$. Denote by $Y=X/X_{1}$. We have the following exact sequence
$$0\rightarrow X_{1}\rightarrow X\rightarrow Y\rightarrow0.$$

Applying $\Sigma_{i_{1}}^{\ast}$ to the above sequence, we have $\Sigma_{i_{1}}^{\ast}X\cong\Sigma_{i_{1}}^{\ast}Y$ because $X_{1}\cong M_{\bfi,1}^{a_{1}}=S_{i_{1}}^{a_{1}}$, $\Sigma_{i_{1}}^{\ast}S_{i_{1}}=0$ and the functor $\Sigma_{i_{1}}^{\ast}$ is right exact.

The filtration (\ref{equ 4}) induces the following filtration of $Y$:
$$0=Y_{1}\subseteq Y_{2}\subseteq \cdots \subseteq Y_{r}=Y,$$
where $Y_{k}\cong X_{k}/X_{1}$ for any $2\leq k\leq r$. In particular, $Y_{k}/Y_{k-1}\cong X_{k}/X_{k-1}\cong M_{\bfi,k}^{a_{k}}$.

By Proposition \ref{prop M_k and replection funtctor}, $M_{\bfi,k}=\Sigma_{i_{1}}\cdots \Sigma_{i_{k-1}}S_{i_{k}}$ for any $1\leq k\leq r$. So for $2\leq k\leq r$, $M_{\bfi,k}$ has trivial $i_{1}$-socle. Then by Lemma \ref{lem restrict exactness}, we know that $\Sigma_{i_{1}}^{\ast}Y$ has a filtration
$$0=Y'_{1}\subseteq Y'_{2}\subseteq \cdots \subseteq Y'_{r}=Y$$
such that $Y'_{k}/Y'_{k-1}\cong (\Sigma_{i_{1}}^{\ast}M_{\bfi,k})^{a_{k}}$.

Note that $\Sigma_{i_{2}}\cdots\Sigma_{i_{k-1}}S_{i_{k}}$ has trivial $i_{1}$-top (Corollary \ref{cor trivial top}). By Lemma \ref{lem prop of ref funt} (ii), we deduce that
$$\Sigma_{i_{1}}^{\ast}M_{\bfi,k}=\Sigma_{i_{1}}^{\ast}\Sigma_{i_{1}}\Sigma_{i_{2}}\cdots\Sigma_{i_{k-1}}S_{i_{k}}\cong
\Sigma_{i_{2}}\cdots\Sigma_{i_{k-1}}S_{i_{k}}.$$

This proves $\Sigma_{i_{1}}^{\ast}Y\in\Lambda_{\bfi'}^{\bfa'}$.
\end{proof}

\begin{lem}\label{lem Z i' a'}
$\mathcal{T}_{i_{1}}\widetilde{e}_{i_{1}}^{\ast\max}(Z_{\bfi}^{\bfa})=Z_{\bfi'}^{\bfa'}$.
\end{lem}
\begin{proof}
For any $X\in\Lambda_{\bfi}^{\bfa}$, since $M_{\bfi,k}\simeq\Sigma_{i_{1}}\cdots\Sigma_{i_{k-1}}S_{i_{k}}$ has trivial $i_{1}$-socle for all $2\leq k\leq r$, we have
$$\soc_{i_{1}}X\cong \soc_{i_{1}}M_{\bfi,i_{1}}=S_{i_{1}}^{a_{1}}.$$

We know that the irreducible component $Z_{\bfi}^{\bfa}$ contains a dense open subset belonging to $\Lambda_{\bfi}^{\bfa}$, thus we have  $\varepsilon_{i_{1}}^{\ast}(Z_{\bfi}^{\bfa})=a_{1}$.

Now by the dual of Proposition 5.5 in \cite{BK}, there exists a dense open subset $U$ of $Z_{\bfi}^{\bfa}$ such that $U\subseteq\Lambda_{\bfi}^{\bfa}$ and for any $X\in U$, $\Sigma_{i_{1}}^{\ast}X$ lies in a dense open subset of $\mathcal{T}_{i_{1}}\widetilde{e}_{i_{1}}^{\ast\max}(Z_{\bfi}^{\bfa})$.

So Lemma \ref{lem Sigma star X} implies $\mathcal{T}_{i_{1}}\widetilde{e}_{i_{1}}^{\ast\max}(Z_{\bfi}^{\bfa})=Z_{\bfi'}^{\bfa'}$.
\end{proof}

We can apply Lemma \ref{lem Z i' a'} successively to $Z_{\bfi}^{\bfa}$ until we reach the unique irreducible component of the variety $\Lambda(0)$ (in fact it is a point), which is the highest weight element $Z_{0}$ in $\mathcal{B}$. Namely we have
\begin{equation}\label{equ 5}
\widetilde{e}_{i_{r}}^{\ast\max}\mathcal{T}_{i_{r-1}}\widetilde{e}_{i_{r-1}}^{\ast\max}\cdots
\mathcal{T}_{i_{1}}\widetilde{e}_{i_{1}}^{\ast\max}Z_{\bfi}^{\bfa}=Z_{0}.
\end{equation}

Now note that the crystal isomorphism $\Psi:\mathscr{B}(\infty)\simeq\mathcal{B}$ maps $b_{0}$ to $Z_{0}$ and commutes with $\widetilde{e}_{i}^{\ast}$ and $\mathcal{T}_{i}$. Comparing (\ref{equ 5}) with Corollary \ref{cor reach b from the hwc}, we complete the proof of Theorem \ref{thm main}.

\bigskip
\par\noindent {\bf Acknowledgments.}
Part of this work was done when the author was visiting the Hausdorff Research Institute for Mathematics (HIM) in Bonn as a participant of the trimester program "On the interaction of representation theory with geometry and combinatorics". The author would like to thank the organizers for their hospitality. He is grateful to Prof. Jan Schr\"{o}er for interesting discussions and comments. He is also very grateful to the anonymous
referee for pointing out two errors in an earlier draft and his/her valuable suggestions.

\bibliographystyle{amsplain}

\begin{thebibliography}{30}
\bibitem{AIRT} C. Amiot, O. Iyama, I. Reiten and G. Todorov, Preprojective algebras and $c$-sortable words, \textit{Proc. London Math. Soc.} 104 (2012), 513-539.

\bibitem{BK} P. Baumann and J. Kamnitzer, Preprojective algebras and MV polytopes, \textit{Represent. Theory} 16 (2012), 152-188.

\bibitem{BKT} P. Baumann, J. Kamnitzer and P. Tingley, Affine Mirkovi\'{c}-Vilonen polytopes, Preprint arXiv:1110.3661.

\bibitem{BFZ} A. Berenstein, S. Fomin and A. Zelevinsky, Parametrizations of canonical bases and totally positive matrices, \text{Adv. Math.} 122 (1996), 49-149.
    
\bibitem{BZ} A. Berenstein and A. Zelevinsky, Tensor product multiplicities, canonical bases and totally positive varieties, \textit{Invent. Math.} 143 (2001), 77-128.

\bibitem{BGP} I. N. Berstein, I. M. Gelfand and V. A. Ponomarev, Coxeter functors and Gabriel's theorem, \textit{Uspehi. Mat. Nauk} 28 (1973), 19-33.

\bibitem{Bo} Bolten, B. ``Spiegelungsfunktoren f\"{u}r pr\"{a}projektive Algebren." Diploma Thesis, Universit\"{a}t Bonn, 2010.

\bibitem{BIRS} A. Buan, O. Iyama, I. Reiten and J. Scott. Cluster structures for 2-Calabi-Yau categories and unipotent groups, \textit{Compositio Math.} 145 (2009), 1035-1079.


\bibitem{GLS1} C. Gei{\ss}, B. Leclerc and J. Schr\"{o}er, Cluster algerba structures and semicanonical bases for unipotent groups, Preprint arXiv:0703039.

\bibitem{GLS2} C. Gei{\ss}, B. Leclerc and J. Schr\"{o}er, Kac-Moody groups and cluster algebras, \textit{Adv. Math.} 228 (2011), 329-433.


\bibitem{Ka1} M. Kashiwara, On crystal bases of the $q$-anlogue of universal enveloping algebras, \textit{Duke Math. J.} 63 (1991), 456-516.

\bibitem{Ka2} M. Kashiwara, The crystal base and Littelmann's refined Demazure character formula, \textit{Duke Math. J.} 71 (1993), 839-858.

\bibitem{KS} M. Kashiwara and Y. Saito, Geometric construction of crystal bases, \textit{Duke Math. J.} 89 (1997), 9-36.

\bibitem{Ki} Y. Kimura, Quantum unipotent subgroups and dual canonical bases, Preprint arxiv:1010.4212.

\bibitem{Lu1} G. Lusztig, Canonical bases arising from quantized enveloping algebras, \textit{J. Amer. Math. Soc.} 3 (1990), 447-498.

\bibitem{Lu2} G. Lusztig, Quivers, pervere sheaves and quantized enveloping algebras, \textit{J. Amer. Math. Soc.} 4 (1991), 365-421.

\bibitem{Lu3} G. Lusztig, Introduction to quantum groups, Birkh\"{a}user, Boston, 1993.

\bibitem{Lu4} G. Lusztig, Braid group actions and canonical bases, \textit{Adv. Math.} 122 (1996), 237-261.

\bibitem{Lu5} G. Lusztig, Semicanonical bases arising from enveloping algebras, \textit{Adv. Math.} 151 (2000), 129-139.

\bibitem{R} C. M. Ringel, The preprojective algebra of a quiver, Algebras and modules II (Geiranger, 1966), CMS Conf. Proc. 24, AMS (1998), 467-480.

\bibitem{Sai} Y. Saito, PBW-basis of quantized universal enveloping algebras, \textit{Publ. RIMS, Kyoto Univ.} 30 (1994), 209-232.

\bibitem{Sav1} A. Savage, Quiver varieties and Demazure modules, \textit{Math. Ann.} 355 (2006), 31-46.

\bibitem{Sav2} A. Savage, Quiver grassmannians, quiver vareties and the preprojective algebra, \textit{Pacific J. Math.} 251 (2011), 393-430.

\end{thebibliography}

\end{document}